\documentclass[]{amsart}
\usepackage{amsmath}
\usepackage{amsfonts}
\usepackage{amssymb}
\usepackage{amsthm}
\usepackage{enumerate}
\usepackage{graphicx,color,xcolor,tikz}
\usepackage{bbold}
\usepackage[utf8]{inputenc}
\usepackage{hyperref}
\hypersetup{
    colorlinks=true,                         
    linkcolor=blue, 
    citecolor=red, 
  } 

\newtheorem{Prop}{Proposition}

\newtheorem{Rem}{Remark}

\newcommand{\p}{\partial}

\renewcommand{\div}{\text{div}}

\title{Derivation of a two-phase flow model accounting for
  surface tension}
\date{\today}

\author{H. Mathis}
\address{Institut Montpelliérain Alexander Grothendieck, Université de
  Montpellier, CNRS, Montpellier, France}
\email{helene.mathis@umontpellier.fr}

\begin{document}
\begin{abstract}
  This paper addresses the derivation of a two-phase flow model
  accounting for surface tension effects, by means of the Stationary
  Action Principle (SAP). 
  The Lagrangian functional, defining the Action,  is
  composed of a kinetic energy, accounting for interface feature, and
  a potential energy.
 The key element of the model lies on the assumption that the
  interface separating the two phases admits its own internal energy,
  satisfying a Gibbs form including  both surface
  tension and interfacial area. Thus surface tension is taken into
  account both in the potential energy and the kinetic one which
  define the Lagrangian functional.
  Applying the SAP allows to build a set of partial
  differential equations modelling the dynamics of the two-phase flow.
  It includes evolution equations of the volume fraction and the
  interfacial area, accounting for mechanical relaxation terms.
  The final model is shown to be well posed (hyperbolicity, Lax
  entropy).
\end{abstract}

\maketitle


\noindent
\textbf{2020 MCS.} 76T05, 76N, 76A02, 80A10 

\tableofcontents

\section{Introduction}
\label{sec:introduction}

The modelling of compressible multiphase flows have been the topic of
a large literature over the past decades, notably for practical
applications such as nuclear safety of pressurized water reactor. In
the context of the loss of coolant accident for instance,
the liquid water refrigerant is submitted to high pressure and
temperature condition, so that a break in the refrigerant circuit
could lead to the appearance of vapor and induce shock and phase
transition waves \cite{bartak90}.
Hence the question is not
only to capture the wave structure but also to get informations on the
different exchanges occuring at the liquid-vapor interface.

These transfers depend strongly on the area of the
interfaces, even when focusing on the large scale description.
Several approaches have been proposed to establish the evolution of
the interfacial area, depending mostly of the scale of description.
When focusing on polydisperse
flows, with many inclusions, bubbles or droplets, the modelling of the
bubbles pulsation requires to keep a small scale description.
For instance in \cite{LHUILLIER2004}, the author proposes a transport
equation based on heuristics of particulate suspensions, assuming that
both phases evolve with distinct velocities. Then focusing
on the small scale, he proposes a second transport equation while
studying the fluctuations of a small interface element. Following this
approach several models have been proposed in a serie of works
\cite{Kokh19, CordesseEtAl20, cordesse20, dibattista:tel}, considering the
one-velocity framework.
The set of bubbles/droplets is described by a probability density function,
which satisfies the so-called Williams-Boltzmann equation. The
distribution function describes then the probability of presence of a
bubble at a certain time and position, which evolves with a given
velocity. It also takes into account topological properties of the
bubble/droplet such as its volume or radius. In a recent
contribution \cite{loison2023}, the authors propose to make the density distribution
depend on topological informations of the interface as well,
considering in particular the level set of the interface and its local
mean curvature. Doing so they manage to get informations on the
interfacial area and provide a full partial differential equations (PDEs) model for the mixture
dynamics, by means of Stationary Action Principle (SAP) applied to a given
Lagrangian functional. Note that the model they get enters the class
of averaged models, in the sense that it does not provide any
information on the topology of the interface, which is solely
indicated by evolution equations of 1) the void fraction of one of the two phases
and 2) the interfacial area.

In the application we have in mind, a precise description of the
interface topology is not mandatory. The high heterogeneity of the
flows, the strong temperature and pressure conditions suggest to
consider averaged models where the interface is depicted implicitly.
However the interface between the
two phases is the locus of all the thermodynamical exchanges, and the
relaxation towards the thermodynamical depends strongly on this area,
especially the relaxation time scales.
This is the approach adopted in \cite{Boukili19} where a convection
equation of the interfacial area is coupled to a barotropic 
three-phase flow model of
Baer-Nunziato type. The equation is endowed with a source term
which cancels as soon as the Weber number (ratio of the momentum over
the surface tension) is greater than a given threshold or when the
relative velocity of the two phases is null.
The interfacial area equation is inspired by the modelling proposed in
\cite{MEIGNEN2014528} for steam explosion simulations.
\medskip

In the latter references the surface tension effect solely
depends on geometrical features which define the kinetic energy of the
bulk and the interface. In all the proposed derivations, the
thermodynamic behaviour of the interface is not considered, and in
particular, the surface tension does not affect the potential energy
of the system. This is precisely this point of vue that we are developing here. 

The derivation of the averaged fluid-interface model is
obtained adapting the Stationary Action Principle as detailled in \cite{BGP23}.
The core of our model relies on the rigorous derivation of the
potential energy involved in the Lagrangian
functional, while the kinetic energy accounts for smale scales, in the
spirit of \cite{Kokh19}.

The originality of our approach is to consider the surface tension
not only as a dynamical feature but also a thermodynamics one.
To do so we come back to classic extensive thermodynamics in the
sense of Gibbs \cite{callen85,Kondepudi} to described as rigorously as possible the
thermodynamical behaviour of the fluid-interface system.
In the case of two
phase flows, this methodology has been used  for instance
in \cite{HS06, HM10, faccanoni12, FM19} and in the quoted references,
leading to thermodynamically consistent multiphase flow models (for
immiscible and miscible mixtures).
The novelty here relies in the fact that the interface is assumed to
be described by an extensive internal energy function, assuming that
the interface has no mass, occupies no volume but is characterized by
its temperature and its (interfacial) area. Thus the
interface is entirely described by its internal energy which satisfies a
Gibbs form involving not only temperature/entropy but also
surface tension/interfacial area variations. Note that this
characterization goes back to \cite{Landau, Kondepudi} and more recently in
\cite{smai2020} to model (multiphase) flows in  porous media.
\medskip

The paper is organized as follows. The Section \ref{sec:therm-modell} is devoted to the
thermodynamical modelling focusing on the extensive and intensive
descriptions of the two phases and the interface. The 
internal energy of the fluid-system is defined  and the associated mixture temperature,
pressure and chemical potential as well. Besides the initial setting ensures that the
interfacial area and
the surface tension appear naturally in the pressure.
When focusing on the characterization of the thermodynamical
equilibrium, some interessant properties arise, especially the fact
that the mechanical equilibrium is depicted by a differential form
involving the volume fraction and the interfacial area.

In Section \ref{sec:deriv-evol-equat} we make use of the
fluid-interface internal energy to define
the potential energy of the Lagrangian functional defining the Action.
A review of the so-called small scale kinetic energies available in the
literature is presented and motivates the choice of our kinetic energy.
Then the SAP leads to the obtention
of a set of PDEs describing the dynamics of the fluid-interface system.
Since the SAP qualifies reversible processes, it guarantees
the conservations of the momentum and the total energy and the
hyperbolicity and symmetrization property of the model. Such properties
are given in Section
\ref{sec:analyse}.

\section{Thermodynamical modelling}
\label{sec:therm-modell}
We consider the fluid-interface system with volume $V$, mass $M$ and entropy
$S$. It is composed of the two immiscible fluids or phases $k=1,2$ with
indices $k=1,2$,
separated by an interface, with index $i$. At each point of this system, local
equilibrium is reached so that each part of the system is depicted by
its own Equation of State (EoS).
In the present section are listed the notations and assumptions for
the fluid
phases and the interface, both in extensive and intensive variables.
Then, the second law of thermodynamics allows us to characterize the
fluid-interface internal energy and the
thermodynamical equilibrium of the system. 

\subsection{Fluid phases}
\label{sec:single-phase}
A phase $k=1,2$ of volume $V_k\geq 0$, entropy $S_k\geq 0$ and mass
$M_k\geq 0$ is entirely described by its extensive internal energy function
$E_k$ which complies with the following assumptions:
\begin{itemize}
\item $(M_k,V_k,S_k)\mapsto E_k(M_k,V_k,S_k)$ is $\mathcal C^2$ on
  $(\mathbb R_+)^3$,
\item  $(M_k,V_k,S_k)\mapsto E_k(M_k,V_k,S_k)$ is convex,
\item $\forall \lambda\in \mathbb R_+^*, \forall (M_k,V_k,S_k)\in
  (\mathbb R_+)^3, E_k(\lambda M_k,\lambda V_k,\lambda S_k)=\lambda
  E_k(M_k,V_k,S_k)$.
\end{itemize}
The last assumption corresponds to the extensive character of the
internal energy function : when doubling the volume, mass and entropy of the system, the extensive
internal energy is doubled as well. This homogeneity property implies
that the extensive internal energy $E_k$ is convex but not
strictly convex.

Some intensive parameters are defined as partial derivatives of $E_k$:
\begin{itemize}
\item the pressure $p_k(M_k,V_k,S_k)=-\p E_k/\p V_k(M_k,V_k,S_k)$,
\item the temperature  $T_k(M_k,V_k,S_k)=\p E_k/\p
  S_k(M_k,V_k,S_k)>0$,
\item the chemical potential $\mu_k(M_k,V_k,S_k)=\p E_k/\p M_k(M_k,V_k,S_k)$,
\end{itemize}
leading to the total differential form
\begin{equation}
  \label{eq:Gibbs_ext}
  \mathrm{d}E_k = T_k\mathrm{d}S_k -p_k\mathrm{d}V_k+\mu_k\mathrm{d}M_k,
\end{equation}
referred as extensive (phasic) Gibbs form in the sequel.
Since the internal energy is extensive, its satisfies the Euler
relation
\begin{equation}
  \label{eq:euler}
  E_k = T_kS_k -p_kV_k+\mu_kM_k.
\end{equation}
Some intensive variables and potentials can be defined while
considering the extensive ones relatively to the mass of the phase $k$.
We introduce the specific volume $\tau_k = V_k/M_k$ and the specific
entropy $s_k= S_k/M_k$ of the phase $k=1,2$. Then the specific
internal energy $e_k(\tau_k,s_k)$ corresponds to a restriction of the
extensive energy:
\begin{equation}
  \label{eq:int_energy}
  e_k(\tau_k,s_k) = E_k(1,\tau_k,s_k).
\end{equation}
The phasic pressure and temperature can be defined as functions
of the intensive variables as well (while keeping the same notations):
\begin{equation}
  \label{eq:int_T_p}
  p_k(\tau_k,s_k)=-\p e_k/\p \tau_k(\tau_k,s_k), \quad
  T_k(\tau_k,s_k)=\p e_k/\p s_k(\tau_k,s_k).
\end{equation}
The intensive potentials comply thus with an intensive differential
(phasic) Gibbs form:
\begin{equation}
  \label{eq:int_gibbs}
  \mathrm{d}e_k = T_k\mathrm{d}s_k-p_k\mathrm{d}\tau_k.
\end{equation}
Note that scaling the extensive Euler relation \eqref{eq:euler} with respect to the
mass $M_k$
gives another definition of the chemical potential $\mu_k$, which turns to be
the Legendre transform of the internal energy $e_k$:
\begin{equation}
  \label{eq:muk}
  \mu_k = e_k-T_ks_k+p_k\tau_k.
\end{equation}

\subsection{The interface}
\label{sec:interface}
The interface separating the two phases is supposed to be sharp and to
have no volume and no mass. Adopting an extensive description, it is
thus characterized by its energy $E_\mathrm{i}$, function of its entropy $S_\mathrm{i}$
and its area $A_\mathrm{i}$. According to the first principle of thermodynamics,
it holds
\begin{equation}
  \label{eq:int_gibbs_i}
  \mathrm{d}E_\mathrm{i} = T_\mathrm{i}\mathrm{d}S_\mathrm{i}+\gamma_\mathrm{i}\mathrm{d}A_\mathrm{i},
\end{equation}
where $\gamma_\mathrm{i} (S_\mathrm{i} A_\mathrm{i})$ is the surface tension and $T_\mathrm{i}(S_\mathrm{i},A_\mathrm{i})$
the interfacial temperature.
The internal energy $E_\mathrm{i}$ being an extensive quantity, its complies
with the Euler relation 
\begin{equation}
  \label{eq:int_euler_i}
  E_\mathrm{i} = T_\mathrm{i}S_\mathrm{i}+\gamma_\mathrm{i}A_\mathrm{i},
\end{equation}
which yields, after differentiating and subtracting
\eqref{eq:int_gibbs_i}, the so-called Gibbs--Duhem relation
\begin{equation}
  \label{eq:int_gibbs_duhem}
  0 = S_\mathrm{i}\mathrm{d}T_\mathrm{i} + A_\mathrm{i}\mathrm{d}\gamma_\mathrm{i}.
\end{equation}
Since the interface has no mass, a way to deduce intensive potentials
is to scale with respect to the volume $V$ of the fluid-interface
system.
This way we introduce the interfacial
density area
\begin{equation}
  \label{eq:a}
  a_\mathrm{i}=A_\mathrm{i}/V,
\end{equation}
while scaling the area $A_\mathrm{i}$ by the volume $V$ of the mixture.

Now, scaling the extensive variables with respect to the interface
area $A_\mathrm{i}$ defines the
interfacial intensive entropy $s_\mathrm{i}=S_\mathrm{i}/A_\mathrm{i}$ and the
interfacial intensive energy $e_\mathrm{i}=E_\mathrm{i}/A_\mathrm{i}$.
Scaling the Euler relation \eqref{eq:int_euler_i} with respect to the
volume $V$ of the mixture, one deduces 
\begin{equation}
  \label{eq:ei}
   e_\mathrm{i} = T_\mathrm{i} s_\mathrm{i} + \gamma_\mathrm{i},
\end{equation}
and  doing so with the interfacial
Gibbs relation \eqref{eq:int_gibbs_i} gives
\begin{equation}
  \label{eq:Ti}
   e_\mathrm{i}'(s_\mathrm{i}) = T_\mathrm{i},
\end{equation}
and
\begin{equation}
  \label{eq:daeI}
 \mathrm{d}(a_\mathrm{i}e_\mathrm{i}) = T_\mathrm{i} \mathrm{d}(a_\mathrm{i}s_\mathrm{i})+\gamma_\mathrm{i}\mathrm{d}a_\mathrm{i}.
\end{equation}
By the definition \eqref{eq:a} of the interfacial density area $a_\mathrm{i}$,
observe that the relation \eqref{eq:int_gibbs_duhem}
gives
\begin{equation}
  \label{eq:int_gam}
  \gamma_\mathrm{i}'(T_\mathrm{i}) = -s_\mathrm{i}(T_\mathrm{i}).
\end{equation}
This derivative relation is not often mentioned in the literature, but
it can be found in \cite{Landau, Kondepudi} for instance.
\subsection{Thermodynamical equilibrium}
\label{sec:therm-equil}

We now consider the fluid system with volume $V$, mass $M$ and entropy
$S$. The two immiscible phases 
$k=1,2$
are separated by the interface of area $A_\mathrm{i}$.
Accounting for the constitutive laws of the two fluid phases and the
interface, we now turn to the characterization of thermodynamical
equilibrium of the whole system. For a given state $(M,V,S,A_\mathrm{i})$ of the
system, different modelling constraints have to be set.
As mentioned before, the two phases are supposed to be immiscible and
that no vacuum appears,
such that the total volume is the sum of the phasic volumes
\begin{equation}
  \label{eq:imm_cons}
  V=V_1+V_2,
\end{equation}
since the interface has no volume.
As the mass conservation of the system is concerned, it holds
\begin{equation}
  \label{eq:mass_cons}
  M=M_1+M_2,
\end{equation}
since the interface has no mass and only mass transfer can
occur between the two phases (and not with the interface).
Finally the homogeneity property of the system entropy states that
\begin{equation}
  \label{eq:ent_cons}
  S=S_1+S_2+S_\mathrm{i}.
\end{equation}
It is convenient
to provide the intensive counterpart of these constraints while
introducing the fractions of presence of each phase, namely the volume fraction
$\alpha_k=V_k/V\in [0,1]$, the mass fraction $y_k=M_k/M\in [0,1]$ and the entropy
fraction $z_k=S_k/S\in [0,1]$, such that
\begin{equation}
  \label{eq:tauk}
  y_k\tau_k = \alpha_k \tau,\quad y_ks_k =z_ks. 
\end{equation}
Then the intensive counterpart of the extensive constraints reads 
\begin{equation}
  \label{eq:int_cons}
  \begin{cases}
    1 = \alpha_1+\alpha_2,\\
    1= y_1+y_2,\\
    1 = z_1+z_2+z_\mathrm{i},
  \end{cases}
\end{equation}
where $z_\mathrm{i}=S_\mathrm{i}/S\in [0,1]$ stands for the entropy fraction of the
interface.

We now turn to the definition of the 
extensive energy of the whole system. It corresponds to the sum of the energies
of each part, namely
\begin{equation}
  \label{eq:mix_energy}
  E (M,V,S,A_\mathrm{i})= E_1(M_1,V_1,S_1)+E_2(M_2,V_2,S_2)+E_\mathrm{i}(S_\mathrm{i},A_\mathrm{i}).
\end{equation}
Using the
 fractions definitions, the total
 derivative of $E$ reads
 \begin{equation*}
   \begin{aligned}
     \mathrm{d}E &=\sum_{k=1}^2 \bigl( T_k\mathrm{d}S_k -p_k\mathrm{d}V_k +\mu_k\mathrm{d}M_k\bigr)+T_\mathrm{i}\mathrm{d}S_\mathrm{i}+\gamma_\mathrm{i}
     \mathrm{d}A_\mathrm{i}\\
     &=\sum_{k=1}^2\bigl( T_kz_k\mathrm{d}S + ST_k\mathrm{d}z_k
     -p_k\alpha_k\mathrm{d}V_k - 
     Vp_k\mathrm{d}\alpha_k\bigr. \\
     &\qquad \bigl.+ y_k\mu_k\mathrm{d}M_k + M\mu_k\mathrm{d}y_k\bigr) \\
     &+ T_\mathrm{i}z_\mathrm{i}\mathrm{d}S+ ST_\mathrm{i}\mathrm{d}z_\mathrm{i}+\gamma_\mathrm{i} a_\mathrm{i}\mathrm{d}V + \gamma_\mathrm{i} V\mathrm{d}a_\mathrm{i}.
   \end{aligned}
 \end{equation*}
 Reorganizing the terms and using the intensive constraints
 \eqref{eq:int_cons}, one obtains
 \begin{Prop}
   The extensive energy satisfies
 \begin{equation}
   \label{eq:diff_mixt_energy}
   \begin{aligned}
     \mathrm{d}E &=(z_1T_1+z_2T_2+z_\mathrm{i}T_\mathrm{i})\mathrm{d}S - (\alpha_1p_1+\alpha_2 p_2
     -a_\mathrm{i}\gamma_\mathrm{i})\mathrm{d}V\\
     &\quad+ y_1(\mu_1-\mu_2)\mathrm{d}M\\
     &\quad+ S ((T_1-T_\mathrm{i})\mathrm{d}z_1 + (T_2-T_\mathrm{i})\mathrm{d}z_2)\\
     &\quad-V ((p_1-p_2)\mathrm{d}\alpha_1- \gamma_\mathrm{i}  \mathrm{d}a_\mathrm{i})\\
     &\quad+M(\mu_1-\mu_2)\mathrm{d}y_1.
   \end{aligned}
 \end{equation}
 As a consequence, the temperature, pressure and chemical
potential of the fluid-interface system have natural definitions in terms of the phasic and interfacial quantities:
\begin{equation}
  \label{eq:mixt_potentials}
  \begin{cases}
    T:= z_1T_1+z_2T_2+z_\mathrm{i}T_\mathrm{i}, \\
    p:=\alpha_1p_1+\alpha_2 p_2
    -a_\mathrm{i}\gamma_\mathrm{i}, \\
    \mu:=y_1\mu_1+y_2\mu_2.
  \end{cases}
\end{equation}
\end{Prop}
In absence of the surface tension, the mixture pressure coincides with
the mixture pressure classic bi-fluid or two-phase models
\cite{kapila}. When accounting for surface tension, the mixture
pressure is 
exactly the one of the  two-phase flow model, derived in
\cite{hillairet:hal-03621307} by homogenization techniques.
This pressure also appears in jump conditions of Euler-Korteweg
system, see \cite{Rohde2012-cemracs}.
The pressure we get is also close to the pressure
term derived
in \cite{Kokh19, CordesseEtAl20} in the context of two-phase flows
with surface tension.

For a given state
$(M,V,E{,A_\mathrm{i}})$, and according to the second principle of
thermodynamics, the thermodynamical equilibrium corresponds to a
minimum of 
the energy $E$ defined in \eqref{eq:mix_energy} under the extensive
constraints \eqref{eq:imm_cons}-\eqref{eq:ent_cons}. Thus, in the interior of the
constraint set, the derivatives
of $E$ with respect to independant variables cancel, leading to a
characterization of the thermodynamical equilibrium in terms of
phasic potentials.
\begin{Prop}
  \label{prop:equi}
  According to the differential form \eqref{eq:diff_mixt_energy}, the
  thermodynamical equilibrium is characterized by
  \begin{equation}
    \label{eq:equi}
    \begin{cases}
      \mu_1=\mu_2,\\
      T_1=T_2=T_\mathrm{i},\\
      \gamma_\mathrm{i} \mathrm{d}a_\mathrm{i} -
      (p_1-p_2)\mathrm{d}\alpha_1=0.
    \end{cases}
  \end{equation}
\end{Prop}
The two first equalities of \eqref{eq:equi} are classic: they denote the thermal
equilibrium in the fluid-interface system and the mass transfer
between the two fluid phases. The last
(differential) relation represents the mechanical equilibrium, and
brings out some comments:
\begin{itemize}
\item As $\gamma_\mathrm{i}=0$, that is for a planar interface, one recovers that
  the
  mechanical equilibrium corresponds to the saturation of the phasic
  pressures $p_1=p_2$ (see for instance \cite{callen85});
\item Assume that the phase 1 occupies a spherical bubble of radius
  $R$. Then its volume is $V_1=4 \pi R^3/3$ and the interfacial area 
  is $A_\mathrm{i}=
  4\pi R^2$.
   On the other hand the differential relation in \eqref{eq:equi} gives
  \begin{equation*}
    \gamma \mathrm{d} \left(\dfrac{A_\mathrm{i}}{V}\right)
    -(p_1-p_2) \mathrm{d} \left(\dfrac{V_1}{V}\right)=0
  \end{equation*}
  Expressing this latter formula  in terms of the radius $R$, it leads
  to the Young-Laplace law
  \begin{equation*}
    p_1-p_2 = \dfrac{2 \gamma_\mathrm{i}}{R}.
  \end{equation*}
  Classically the  Young-Laplace law involves the mean curvature which
  corresponds here to the inverse of the radius.
  When the radius tends to $+\infty$, the surface becomes planar and
  one recovers the equality of the phasic pressures.
\end{itemize}

 \subsection{Another characterization of the thermodynamical
  equilibrium using free energies}
\label{sec:grand_canonical}

In the context of two-phase flows in porous media,
Sma\"i proposed in \cite{smai2020} to minimize the free energy of the mixture
instead of minimizing the energy. The advantage is that the free
energy of the mixture (also called canonical grand potential in
the porous media framework) is solely a function of the mixture
temperature and the phasic pressures.

In the extensive framework, the free energy $\Omega_k$ of the phase $k=1,2$ is
defined as the (total) Legendre transform of the energy $E_k$:
\begin{equation}
  \label{eq:Omegak}
  \Omega_k = E_k -T_kS_k-\mu_kM_k.
\end{equation}
Differentiating \eqref{eq:Omegak} and using the Gibbs relation
\eqref{eq:Gibbs_ext} give
\begin{equation}
  \label{eq:dOmegak}
  \mathrm{d}\Omega_k = -p_k\mathrm{d}V_k-S_k\mathrm{d}T_k-M_k\mathrm{d}\mu_k.
\end{equation}
In order to introduce intensive potential, it is convenient to scale
with respect to the volume $V_k$ to define the intensive (volumic) free
energy
$\omega_k = \Omega_k/V_k$.
Then using the definitions \eqref{eq:tauk}, it holds
\begin{equation*}
  \begin{aligned}
    \alpha_k \omega_k = \dfrac{e_k}{\tau_k}\alpha_k - T_k
    \dfrac{s_k}{\tau_k}\alpha_k-\dfrac{\mu_k}{\tau_k}\alpha_k,
  \end{aligned}
\end{equation*}
from which one deduces the following differential form
\begin{equation}
  \label{eq:d_akomegak}
  \mathrm{d}(\alpha_k \omega_k) = -p_k \mathrm{d}\alpha_k - \dfrac{\alpha_k
    s_k}{\tau_k}\mathrm{d}T_k - \dfrac{\alpha_k}{\tau_k}\mathrm{d}\mu_k,
\end{equation}
using the intensive relations \eqref{eq:int_gibbs} and \eqref{eq:muk}.

As the interfacial potentials are concerned, the free energy is
defined as 
a partial Legendre transform of the energy $E_\mathrm{i}$, namely
\begin{equation}
  \label{eq:OmegaI}
  \Omega_\mathrm{i} = E_\mathrm{i}-T_\mathrm{i}S_\mathrm{i}.
\end{equation}
Using the interfacial Gibbs relation \eqref{eq:int_gibbs_i}, it yields
\begin{equation}
  \label{eq:dOmegaI}
  \mathrm{d}\Omega_\mathrm{i} = -S_\mathrm{i}\mathrm{d}T_\mathrm{i}+\gamma_\mathrm{i}\mathrm{d}A_\mathrm{i}.
\end{equation}
The appropriate intensive free energy is deduced by scaling with
respect to the interfacial area: $\omega_\mathrm{i}=\Omega_\mathrm{i}/A_\mathrm{i}$.
Hence the intensive free energy reads
\begin{equation}
  \label{eq:aomegai}
  a_\mathrm{i}\omega_\mathrm{i} = a_\mathrm{i}e_\mathrm{i}-T_\mathrm{i} a_\mathrm{i}s_\mathrm{i},
\end{equation}
whose differential is
\begin{equation}
  \label{eq:daomegaI}
  \mathrm{d}(a_\mathrm{i} \omega_\mathrm{i} )= a_\mathrm{i}s_\mathrm{i} \mathrm{d}T_\mathrm{i} +\gamma_\mathrm{i} \mathrm{d}a_\mathrm{i},
\end{equation}
according to \eqref{eq:daeI}.

For the fluid-interface system, the 
extensive free energy corresponds to the sum of the phasic and
interfacial free energies
\begin{equation*}
  \Omega =\Omega_1+\Omega_2+\Omega_\mathrm{i},
\end{equation*}
whose intensive formulation is
\begin{equation}
  \label{eq:int_mixt_omega}
  \omega=\alpha_1\omega_1+\alpha_2 \omega_2+ a_\mathrm{i} \omega_\mathrm{i}.
\end{equation}
According to the differentials \eqref{eq:d_akomegak} and
\eqref{eq:daomegaI}, it yields
\begin{equation}
  \label{eq:d_omega_mix}
  d\omega = -(p_1-p_2)\mathrm{d}\alpha_1 + \gamma_\mathrm{i} \mathrm{d}a_\mathrm{i} -
  \alpha_1\mathrm{d}p_1 -\alpha_2\mathrm{d}p_2 + a_\mathrm{i}s_\mathrm{i}\mathrm{d}T_\mathrm{i},
\end{equation}
since $\alpha_1+\alpha_2=1$.

At thermodynamical equilibrium, the phasic and interfacial potentials
agree with \eqref{eq:equi}. As a consequence the thermodynamical
equilibrium has to comply with an equation of state compatible with
\begin{equation}
  \label{eq:4}
  d\omega =-\alpha_1\mathrm{d}p_1 -\alpha_2\mathrm{d}p_2 + as_\mathrm{i}\mathrm{d}T.
\end{equation}
This differential implies that the equation of state of the
fluid-interface system
can be expressed in term $\omega$, seen as a function of
$T$ and $p_k$, $k=1,2$, such that
\begin{equation*}
  a s_\mathrm{i} = \dfrac{\p \omega}{\p T}(T,p_1,p_2),\quad \alpha_k = - \dfrac{\p \omega}{\p p_k}(T,p_1,p_2).
\end{equation*}
The idea of Sma\"i is to take advantage of this alternative description of the
thermodynamical equilibrium to get rid of the complex description of
the interface: from the equation of state $\omega(T,p_1,p_2)$, one
recovers all the information relative to the interface, without
explicitly computing the interfacial area.

This approach is not developed here since we precisely want to derive
an evolution equation of the interfacial area.

\subsubsection{Potential energy candidate}
\label{sec:potent-energy-cand}

We can now turn to the definition of the potential energy which will be used
in the Lagrangian formulation.
A natural proposition would be to consider the intensive mixture
internal energy when scaling the extensive energy \eqref{eq:mix_energy} by
the total mass $M$.
Then for a given intensive state $(\tau,s,a_\mathrm{i})$, accounting for the
intensive constraints \eqref{eq:int_cons},
the intensive energy would read
\begin{equation}
  \label{eq:int_mixt_energy0}
  e(\tau,s,a_\mathrm{i},(y_k)_k,(\alpha_k)_k, (z_k)_k,z_\mathrm{i}) = y_1
  e_1(\tau_1,s_1)+y_2(\tau_2,s_2)+ a_\mathrm{i} \tau
  e_\mathrm{i}\left(\dfrac{z_\mathrm{i}}{a}\dfrac{s}{\tau}\right), 
\end{equation}
with notations \eqref{eq:tauk} of the phasic quantities.

However it turns out that this choice of variables is not appropriate.
Indeed
it is not clear how the entropy fractions $z_k$ evolve along
trajectories when applying the Stationary Action Principle whereas specific
entropies are conserved along trajectories, according to \cite{GG99,GS02}.
Hence it is more convenient to express the intensive 
fluid-interface internal energy as a function of the intensive entropies $s$ and
$s_k$, $k=1,2$ rather than using the entropy fractions $z_k$. 

In the sequel, we choose to express the intensive energy as a
function of
\begin{equation}
  \label{eq:Btilde}
  \tilde{\mathbf{B}}=\{\rho, s,s_1, s_2, a_\mathrm{i}, y, \alpha\},
\end{equation}
where $\rho=1/\tau$ denotes the mixture density and $y:=y_1$ and
$\alpha:=\alpha_1$. It reads then
\begin{equation}
  \label{eq:int_mixt_energy}
  \begin{aligned}
    e( \tilde{\mathbf{B}}) =& y e_1\left(\dfrac{\alpha}{y
        \rho},s_1\right)+(1-y)e_2\left(\dfrac{1-\alpha}{(1-y)
        \rho},s_2\right)\\
    &+ \dfrac{a_\mathrm{i}}{\rho}
    e_\mathrm{i}\left(\dfrac{s-ys_1-(1-y)s_2}{a_\mathrm{i}}\rho\right).
  \end{aligned}
\end{equation}
Observe that one makes use of the extensive relation
\eqref{eq:ent_cons} on the entropies to express the interfacial entropy $s_\mathrm{i}$ as a
function of $s$, $s_1$ and $s_2$, namely
\begin{equation}
  \label{eq:intensive_entropies}
  s= y s_1 + (1-y) s_2 + \dfrac{a_\mathrm{i}}{\rho}s_\mathrm{i}.
\end{equation}
\begin{Rem}
  \label{rem:dissip}
  In \cite[paragraph 2.1.3.3]{cordesse20}, the author points out the
  importance of the choice of 
  variables on which the Lagrangian functional depends.
  This is also emphasized in the work of Gavrilyuk \cite{Gavrilyuk2020-stbg}.
  Indeed if specific entropies are convenient variables
  for computations, the fact that they are conserved along
  trajectories prohibits any
  interaction between the phases. The fluid-interface entropy will
  also be
  conserved since only reversible processes can be
  depicted by the SAP.  However it is possible to add relaxation
  source terms \textit{a posteriori}, 
  in agreement with the second law of thermodynamics.
  See \cite[Paragraph 3.5]{dibattista:tel} for a presentation of the
  method when dissipation is due to pulsating behaviour of bubbles in
  two-phase flows.
\end{Rem}

\section{Derivation of the evolution equations by means of Stationary
  Action Principle}
\label{sec:deriv-evol-equat}

Accounting for the previous characterization of the thermodynamical
equilibrium, we now turn to the modelling of the fluid dynamics.
The objective is to derive the Euler-type equations satisfied
by the fluid-interface system using the Stationary Action
Principle, following the serie of works \cite{GG99,GavrilyukSaurel02,Drui_these,Kokh19, 
CordesseEtAl20, dibattista:tel,loison2023}.

We focus on homogeneous two-phase flows, in the sense that
the two phases evolve with the same velocity field $\mathbf{u}\in
\mathbb R^3$.
Note that considering distinct velocities is possible as in
\cite{GavrilyukSaurel02}.

The variational approach and the Hamilton's principle of stationary
action rely on the definition of an appropriate
Lagrangian $L$. This Lagrangian is the difference of a kinetic energy and a
potential energy.
The potential energy we propose to consider has been derived in the previous
section, see \eqref{eq:int_mixt_energy}.
As far as the kinetic one is concerned, a small review of recent
models is given in Section \ref{sec:T}, focusing on the so-called \textit{two-scale
kinetic} modelling brought forward  in \cite{Drui_these, cordesse20}.

In Section \ref{sec:L_assum} are stated the main lines of the SAP as well as the
additional assumptions we make (total and partial mass conservations
for instance). As a result is presented the final set of equations, in
its rough form.

\subsection{A non-exhaustive review of kinetic energy}
\label{sec:T}

Recent references takle the derivation of the kinetic energy, motivated by
the initial works of Gavrilyuk and coauthors \cite{GG99,GavrilyukSaurel02}. In the latter
propositions, the kinetic energy $L_{kin}$ is composed of a classic bulk
energy linked to the translational motion of the
fluid and a small scale contribution $T_{pulse}$.
For instance in \cite{GavrilyukSaurel02}, considering distinct
velocities for both the phases $k=1,2$ and the interface, it yields
\begin{equation*}
  L_{kin} = \sum_{k=1}^2 \rho_k \dfrac{|\mathbf{u}_k|^2}{2} +T_{pulse},
\end{equation*}
where $\mathbf{u}_k$ stands for the velocity field of the phase $k$
and
\begin{equation*}
  T_{pulse} =\dfrac m 2
  \left( \dfrac{\mathrm D_\mathrm{\mathrm{i}} \alpha}{\mathrm D t}\right)^2,
\end{equation*}
where $\dfrac{\mathrm D_\mathrm{\mathrm{i}} \cdot}{\mathrm D t}$ is the material
derivative associated to the velocity of the interface $\mathbf{u}_\mathrm{i}$,
namely
$\dfrac{\mathrm D_\mathrm{i} \cdot}{\mathrm D t} = \p_t \cdot + \mathbf{u}_\mathrm{i}
\cdot \nabla_{\mathbf{x}} \cdot$.

According to the authors, the second term is a pulsation kinetic
energy, where the coefficient $m$ and the
interfacial velocity $\mathbf u_\mathrm{i}$ are given by appropriate closure laws.
Considering a one-velocity model, Drui proposes in \cite{Drui_these}
to consider $T_{pulse} = \dfrac 1 2 \nu(\alpha) |\mathrm D_t\alpha|^2$. The
function $\nu$ corresponds to the inertia associated with the motion of
the interface which depends on the volume fraction $\alpha$ only.
Another improvement is introduced by Cordesse \cite{cordesse20, Kokh19},
where the function $\nu$ is a function of the interfacial area, namely
$T_{pulse} = \dfrac 1 2 m \dfrac{|\mathrm D_t\alpha|^2}{a_\mathrm{i}^2}$.
Here and in the sequel, the material derivative is defined using the
common velocity field $\mathbf{u}$
\begin{equation*}
 \mathrm D_t \cdot  =  \p_t \cdot + \mathbf{u}\cdot \nabla_{\mathbf{x}}  \cdot.
\end{equation*}
This last expression of $T_{pulse}$ is derived from geometrical
considerations: when the
interface is subjected to a small displacement, the interfacial area
$a_\mathrm{i}$ and the volume fraction $\alpha$ vary as well, and the
relationships between these quantities involve the local curvature of
the interface and the surface tension parameter, see \cite[Chapter
3]{cordesse20} for more details.
Finally in \cite{dibattista:tel} the function $\nu$ is no longer an
explicit function of $\alpha$ or $a_\mathrm{i}$. The pulsating energy reads
\begin{equation*}
T_{pulse} = \dfrac 1 2 \nu(\alpha,a_\mathrm{i}) {|\mathrm D_th|^2},
\end{equation*}
where $h$ is the local deformation of the interface, which satisfies
differential relations involving the interfacial area, the local
curvature and the volume fraction.

Among all the propositions, what is mandatory is to make the kinetic
energy $L_{kin}$ depends on $\mathrm D_t \alpha$, otherwise there will we be no hope to
get an evolution equation on $\alpha$. For the same reason and
because we want an evolution equation of the interfacial area density,
we propose
to consider also a term involving $\mathrm D_t a_\mathrm{i}$:
\begin{equation}
  \label{eq:kin_energy_T}
 L_{kin} = \dfrac{1}{2}\rho |\mathbf{u}|^2 + \dfrac{m}{2}|\mathrm D_t \alpha|^2 +
\dfrac{\nu}{2}|\mathrm D_t a_\mathrm{i}|^2,  
\end{equation}
where $m$ and $\nu$ are constants (with the appropriate dimensions,
namely $m [kg\cdot m ^{-1}]$ and $\nu[kg \cdot m]$).
Doing so ensures to get an evolution equation of the interfacial
area density, without considering any additional quantities as local
curvature or interface displacement as in \cite{cordesse20, Kokh19}.

\subsection{The Lagrangian functional and additional assumptions}
\label{sec:L_assum}
We introduce the vector of variables $\mathbf{B}$
\begin{equation}
  \label{eq:B}
  \mathbf{B}:=\{\rho, s , s_1, s_2, a_\mathrm{i}, y, \alpha, \mathbf{u}, \mathrm D_t
  \alpha, \mathrm D_t a_\mathrm{i}\},
\end{equation}
which corresponds to the vector $\tilde{\mathbf{B}}$, defined in
\eqref{eq:Btilde}, completed by the variables involved in the kinetic
energy $L_{kin}$, that are $\mathbf{u}$, $\mathrm D_t
\alpha$ and $\mathrm D_t a_\mathrm{i}$.

The Lagrangian $L$, function of $\mathbf{B}$, is the difference of the
kinetic and the potential
contribution
\begin{equation}
  \label{eq:L}
  L(\mathbf{B}) = L_{kin} - L_{pot},
\end{equation}
where $ L_{kin} (\mathbf{B})$ is defined in \eqref{eq:kin_energy_T} and 
$L_{pot} (\mathbf{B}) = \rho e(\tilde{\mathbf{B}})$, with
$e(\tilde{\mathbf{B}})$ defined in \eqref{eq:int_mixt_energy}.

Before going further with the variational method, we make additional
assumptions that govern the fluid-interface system.
First we assume masses conservation, in the sense that
\begin{equation}
  \label{eq:mass_conserv}
  \begin{aligned}
    &\p_t \rho + \div_{\mathbf{x}} (\rho \mathbf{u}) = 0,\\
    &\mathrm D_t y =0.
  \end{aligned}
\end{equation}
One emphasizes that although the modelling presented in Section \ref{sec:therm-equil} allows mass exchange between the
two phases, it is not the case here. This is due to the fact that SAP
is valid for reversible processes only.
For the same reasons, we also assume that the specific entropies are conserved along
trajectories
\begin{equation}
  \label{eq:entrop_cons}
 \mathrm D_t s=0, \qquad \mathrm D_t s_k=0, \qquad k=1,2,
\end{equation}
following \cite{GG99, GavrilyukSaurel02, cordesse20}. 
Notice that, since the specific phasic
entropies are
conserved, the interface intensive entropy (which is relative to the
interfacial area $A_\mathrm{i}$ and not
to mass $M$) is not conserved along trajectories but satisfies
\begin{equation*}
  \mathrm D_t (s_\mathrm{i} a_\mathrm{i} \tau)=0,
\end{equation*}
that is to say, using extensive variables,
$s_\mathrm{i} a_\mathrm{i} \tau =S_\mathrm{i}/M$ is constant along trajectories.

\subsection{Variational principle}
\label{sec:LAP}

This paragraph recalls the classic lines of the Stationary Action
Principle, whose application to the two-phase flow modeling has been the
subjects of numerous works, including
\cite{Bedford80, Bedford21, Geurst86, Gouin90, BGP23}.
See also \cite{dibattista:tel} for a
synthetic presentation of the method and an
overview of the technic in the two-fluid framework.

Consider a volume $\omega(t)\in \mathbb R^3$ occupied by
the fluid-interface system for time $t\in [t_1,t_2]$ and denote $\Omega=\{(t, \mathbf{x})\in
\times [t_1,t_2]\times \mathbb R^d|\; \mathbf{x}\in \omega(t), \, t_1\leq t\leq t_2\}$.
Following Section \ref{sec:L_assum}, we assume the flow to  be fully characterized by
the quantities $(t,\mathbf{x})\mapsto \mathbf{B}$ and by the
constitutive constraints \eqref{eq:mass_conserv}-\eqref{eq:entrop_cons}.
We now define the Hamiltonian Action as the space-time integral of the
Lagrangian functional \eqref{eq:L}
\begin{equation}
  \label{eq:action}
  A(\mathbf{B}) = \int_\Omega L(\mathbf{B})(\mathbf{x},t) \mathrm d
  \mathbf{x}\mathrm d t,
\end{equation}
and apply the Stationary Action principle.
If $(t,\mathbf{x}) \mapsto \mathbf{\bar B}$ is a physically relevant
transformation of the system,
it  is the solution of a variational problem leading to a PDE
system. The methodology is to
consider a family of perturbation
$(t,\mathbf{x},\zeta)\mapsto \mathbf{B}_\zeta$ of $\mathbf{\bar B}$,
parametrized by $\zeta\in [0,1]$ such that
\begin{itemize}
\item the physical path is obtained when $\zeta=0$:
  \begin{equation*}
    \mathbf{B}_\zeta(t,\mathbf{x},\zeta=0) = \mathbf{\bar B} (t,\mathbf{x}),
  \end{equation*}
\item $\mathbf{B}_\zeta$ satisfies the conservation constraints
  \eqref{eq:mass_conserv} and \eqref{eq:entrop_cons} for all $\zeta\in
  [0,1]$,
\item $\mathbf{B}_\zeta(t,\mathbf{x},\zeta)= \mathbf{\bar B} (t,
  \mathbf{x})$ for $(t,\mathbf{x},\zeta) \in \partial\Omega\times[0,1]$.
\end{itemize}
The Stationary Action Principle states that $ \mathbf{\bar B} $ is
physically relevant if it is a stationary point of 
$\zeta\mapsto  A(\mathbf{B}_\zeta)$, that is
\begin{equation}
  \label{eq:Action!}
  \dfrac{\mathrm d A (\mathbf{B}_\zeta)}{\mathrm d\zeta}(0) = 0.
\end{equation}
This stationary condition yields the governing set of PDEs of motion
without dissipative process.
For $b\in \mathbf{\bar B}$,  denoting
\begin{equation*}
  \delta_\zeta b(t,\mathbf{x})  = \left(\dfrac{\p b_\zeta }{\p \zeta}
  \right)_{| t,\mathbf x}(t,\mathbf{x},\zeta=0)
\end{equation*}
a family of infinitesimal transformations, 
the identity \eqref{eq:Action!} reads
\begin{equation}
  \label{eq:SAP}
  \dfrac{\mathrm d A (\mathbf{B}_\zeta)}{\mathrm d\zeta}(\zeta=0) = \displaystyle\int_{\Omega}
  \sum_{b\in \mathbf B} \dfrac{\p L}{\p b} \delta_\zeta b \mathrm d\mathbf{x} \mathrm
  d t.
\end{equation}
Infinitesimal
variations are related through the conservation
principles \eqref{eq:mass_conserv} and \eqref{eq:entrop_cons} (see
\cite{G11} and \cite{BGP23} for detailed computations)
\begin{itemize}
\item variation of density
  \begin{equation}
    \label{eq:var_mass}
    \delta \rho = -\div_{\mathbf{x}} (\rho\delta\mathbf{x}),
  \end{equation}
  where $\delta\mathbf{x}$ denotes the infinitesimal displacement
  $(t,\mathbf{x}) \mapsto \delta \mathbf{x}$ around the physical path,
  which complies with $\delta\mathbf{x}_{|t=t_1} =
  \delta\mathbf{x}_{|t=t_2}=0$ and $\delta\mathbf{x}_{|\partial \omega(t)}=0$. 
\item variation of velocity
  \begin{equation}
    \label{eq:var_u}
    \delta \mathbf{u}  = \mathrm D_t (\delta\mathbf{x}) -\nabla_{\mathbf{x}}
    \mathbf{u}\cdot \delta\mathbf{x},
  \end{equation}
  
\item conservation along trajectories of the fluid specific entropies and the
  mass fraction
  \begin{equation}
    \label{eq:var_s}
    \delta b = -\nabla_{\mathbf{x}} b \cdot \delta \mathbf{x}, \quad \text{for }
    b\in \{ s,y,z_1,z_2\}.
  \end{equation}
\end{itemize}

We now list all the contributions in \eqref{eq:SAP}.
\begin{itemize}
\item Density contribution: using the mass conservation
  \eqref{eq:mass_conserv}, one has
  \begin{equation}
    \label{eq:Arho1}
    \begin{aligned}
      \int_\Omega \dfrac{\p L}{\p \rho} \delta \rho
      \mathrm d\mathbf{x} \mathrm d t &= - \int_\Omega \dfrac{\p
        L}{\p \rho} \text{div}_{\mathbf{x}}(\rho\delta \mathbf{x}) \,\mathrm{d}\mathbf{x} \mathrm{d}t \\
      &=
      \int_\Omega\rho \nabla_{\mathbf{x}} \left(
        \dfrac{\p L}{\p \rho}\right)
      \cdot \delta \mathbf{x}
      \,\mathrm{d}\mathbf{x} \mathrm{d}t,
    \end{aligned}
\end{equation}
by integration by parts. In order to make the partial Legendre transform of
$L$ with respect to $\rho$ (written here as a function of $\mathbf{B}$)
\begin{equation}
  \label{eq:L*}
  L^{*,\rho} (\mathbf{B})= \rho \dfrac{\p L}{\p \rho} - L (\mathbf{B})
\end{equation}
appear, one develops
\begin{equation}
  \label{eq:Arho2}
  \begin{aligned}
      &\int_\Omega
  \dfrac{\p L}{\p \rho} \delta \rho \,\mathrm d\mathbf{x }
  \mathrm{d} t \\
  =& \int_\Omega\left[ \nabla_{\mathbf{x}}\left(\rho \dfrac{\p L}{\p
    \rho}\right) - \dfrac{\p L}{\p
    \rho} \nabla_{\mathbf{x}} \rho \right]\cdot \delta \mathbf{x} \,\mathrm{d}\mathbf{x}
  \mathrm{d}t\\
   =& \int_\Omega \left[\nabla_{\mathbf{x}} \left(\rho \dfrac{\p L}{\p
    \rho}-L\right)+ \nabla_{\mathbf{x}}  L - \dfrac{\p L}{\p
    \rho}\nabla_{\mathbf{x}}  \rho \right] \cdot \delta \mathbf{x} \,\mathrm{d}\mathbf{x}
  \mathrm{d}t\\
  =& \int_\Omega \left( \nabla_{\mathbf{x}} 
    L^{*,\rho} + \sum_{\substack{b \in \mathbf{B}\\ b\neq \rho}}
    \nabla_{\mathbf{x}}  b \dfrac{\p L}{\p b}\right) \cdot \delta \mathbf{x} \,\mathrm{d}\mathbf{x}
  \mathrm{d}t.
\end{aligned}
\end{equation}
\item Velocity contribution: according to \eqref{eq:var_u}, it holds
  \begin{equation}
    \label{eq:eqAu1}
      \int_\Omega
      \dfrac{\p L}{\p \mathbf{u}} \delta \mathbf{u} \,\mathrm d\mathbf{x} \mathrm dt
      =   \int_\Omega
      \dfrac{\p L}{\p \mathbf{u}} \left(\mathrm D_t (\delta \mathbf{x})
        -\nabla_{\mathbf{x}}  \mathbf{u} \cdot \delta \mathbf{x}\right) \mathrm{d}\mathbf{x} \mathrm{d}t.
    \end{equation}
    By definition of the material derivative $\mathrm D_t \cdot$ and using an
    integration by part, it holds
  \begin{equation}
    \label{eq:eqAu2}
    \begin{aligned}
      &\int_\Omega \dfrac{\p L}{\p \mathbf{u}} \delta
      \mathbf{u} \,\mathrm d\mathbf{x} \mathrm dt \\
      =& \int_\Omega \dfrac{\p L}{\p \mathbf{u}} \left[\p_t (\delta
        \mathbf{x})+ \mathbf{u}\cdot \nabla_{\mathbf{x}} (\delta\mathbf{x})-\nabla_{\mathbf{x}} 
        \mathbf{u} \cdot
      \delta \mathbf{x}\right] \mathrm{d}\mathbf{x} \mathrm{d}t\\
      =& -\int_\Omega
      \left( \p_t \left(\dfrac{\p L}{\p
            \mathbf{u}} \right)+ \text{div}_{\mathbf{x}}  \left(\mathbf{u} \dfrac{\p L}{\p
              \mathbf{u}}\right) +\dfrac{\p L}{\p \mathbf{u}} \nabla_{\mathbf{x}} 
          \mathbf{u}\right) \cdot \delta \mathbf{x}\, \mathrm{d}\mathbf{x}\mathrm{d}t.
    \end{aligned}
    \end{equation}
  \item Contributions of conserved quantities along trajectories:
    using \eqref{eq:var_s}, it holds for $b\in \{s, s_1,s_2,y\}$
    \begin{equation}
      \label{eq:Ab1}
        \int_{\Omega}
      \dfrac{\p L}{\p b} \delta b \, \mathrm d\mathbf{x} \mathrm dt
      =-\int_\Omega \dfrac{\p L }{\p b} \nabla_{\mathbf{x}}  b \cdot
      \delta \mathbf{x} \,\mathrm{d}\mathbf{x} \mathrm{d}t.
    \end{equation}
    
  \item Contributions in $\alpha$ and $a_\mathrm{i}$: the variation of the
    volume fraction $\alpha$ is not subjected to any
    constraint. Doing so ensures to get an evolution equation on
    $\alpha$. Therefore the variation $\delta \alpha$, involved with
    the family of transformations of the medium, is arbitrary.
    The same holds for the interfacial area density $a_\mathrm{i}$. Besides, the
    fact that they evolve independently will yield separate equations
    for the volume fraction and the interfacial area density.

  \item Contributions in $\mathrm D_t \alpha$ and $\mathrm D_t a_\mathrm{i}$: the variations of
    $\mathrm D_t \alpha$ (resp. $\mathrm D_t a_\mathrm{i}$) is related to the variation of
    $\alpha$ (resp. $a_\mathrm{i}$).
    According to \cite{Kokh16}, it holds, for any functions $f$ and $g$,
    it holds
    \begin{equation}
    \label{eq:prop8}
    \begin{aligned}
      \int_\Omega g \,\delta (\mathrm D_t f) \,\mathrm d\mathbf{x}
      \mathrm dt= &-\int_\Omega \bigl(\p_t g + \text{div}_{\mathbf{x}} 
      (\mathbf{u}g)\bigr) \delta f \,\mathrm d\mathbf{x} \mathrm dt \\
      &-
      \int_\Omega \left[(\p_t g+ \text{div}_{\mathbf{x}} (\mathbf{u}g)) \nabla_{\mathbf{x}}  f + g
        \nabla_{\mathbf{x}}  (D_t f)\right] \cdot \delta \mathbf{x}\,\mathrm
      d\mathbf{x} \mathrm dt.
    \end{aligned}
  \end{equation}
  Thus using \eqref{eq:prop8} with $g=\dfrac{\p L}{\p
      (\mathrm D_t\alpha)}=:M$ and $f=\mathrm D_t\alpha$ gives
    \begin{equation}
      \label{eq:ADtalpha}
      \begin{aligned}
        &\int_\Omega \dfrac{\p L}{\p (\mathrm D_t \alpha)} \delta
        (\mathrm D_t \alpha) \,\mathrm d\mathbf{x} \mathrm dt \\
        &= -
        \int_\Omega (\p_t M+ \text{div}_{\mathbf{x}} (M\mathbf{u}))\delta
        \alpha \, \mathrm d\mathbf{x} \mathrm dt \\
        &- \int_\Omega ((\p_t M+
        \text{div}_{\mathbf{x}} (M\mathbf{u}))\nabla_{\mathbf{x}}  \alpha + M \nabla_{\mathbf{x}}  (\mathrm D_t
        \alpha)) \cdot \delta \mathbf{x}\, \mathrm d\mathbf{x} \mathrm dt.
      \end{aligned}
    \end{equation}
    Analogously it holds with $g=\dfrac{\p L}{\p (\mathrm D_ta_\mathrm{i})}=:P$ and $f=\mathrm D_ta_\mathrm{i}$
    \begin{equation}
      \label{eq:ADtai}
      \begin{aligned}
        &\int_\Omega \dfrac{\p L}{\p (\mathrm D_t a_\mathrm{i})} \delta
        (\mathrm D_t a_\mathrm{i}) \,\mathrm d\mathbf{x} \mathrm dt \\
        &= -
         \int_\Omega (\p_t P+ \text{div}_{\mathbf{x}} (P\mathbf{u}))\delta a_\mathrm{i} \,\mathrm d\mathbf{x} \mathrm dt\\
         &- \int_\Omega ((\p_t P+
         \text{div}_{\mathbf{x}} (P\mathbf{u}))\nabla_{\mathbf{x}}  a_\mathrm{i} + P \nabla_{\mathbf{x}}  (\mathrm D_t
         a_\mathrm{i}))\cdot \delta \mathbf{x} \, \mathrm{d}\mathbf{x} \mathrm{d}t.
      \end{aligned}
    \end{equation}
\end{itemize}

Finally gathering \eqref{eq:Arho2}, \eqref{eq:eqAu2}, \eqref{eq:ADtalpha} and \eqref{eq:ADtai}
gives
\begin{equation*}
  \int_{\Omega(0)} [A_\alpha\delta \alpha+ A_{a_\mathrm{i}}\delta
  a_\mathrm{i}+ A_\mathbf{u}\delta \mathbf{x} ]\,\mathrm{d}\mathbf{x}\mathrm{d}t=0.
\end{equation*}
where
\begin{equation}
  \label{eq:deltaA1}
  \begin{cases}
    A_\alpha = \p_t M + \text{div} (M\mathbf{u}) -\dfrac{\p L}{\p
      \alpha}, & \text{with } M = \dfrac{\p L}{\p
      (\mathrm D_t\alpha)},\\
   A _{a_\mathrm{i}}= \p_t P + \text{div} (P\mathbf{u}) -\dfrac{\p L}{\p
      a_\mathrm{i}}, & \text{with } P= \dfrac{\p L}{\p
      (\mathrm D_ta_\mathrm{i})},\\
    A_\mathbf{u}= \p_t  K + \text{div} (K\mathbf{u}) -\nabla L^{*,\rho}, &
    \text{with } K= \dfrac{\p L}{\p \mathbf{u}}.
  \end{cases}
\end{equation}
Note that to express the term $C$, one makes use of the terms $A$ and
$B$.

Since one assumes the infinitesimal displacement and the variations of
volume fraction and interfacial area density to be independent,
the SAP applied to the Lagrangian
energy $L$ yields the equations of motion given by
\begin{equation*}
  A_\alpha=0, \quad A _{a_\mathrm{i}}=0, \quad A_\mathbf{u}=0.
\end{equation*}

\section{Final system and properties}
\label{sec:analyse}

As a result of the Stationary Action Principle, one obtains the following
set of equations describing the time evolution of the fluid-interface
system governed by the Lagrangian $L$.
It reads
\begin{equation}
  \label{eq:final_set}
  \begin{cases}
    \p_t M + \text{div} (M\mathbf{u}) -\dfrac{\p L}{\p
      \alpha} =0,\\
    \p_t P + \text{div} (P\mathbf{u}) -\dfrac{\p L}{\p
      a_\mathrm{i}}=0,\\
    \p_t  K + \text{div} (K\mathbf{u}) -\nabla L^{*,\rho}=0,
  \end{cases} 
\end{equation}
where $L^{*,\rho} $ the partial Legendre transform of $L$ defined in
\eqref{eq:L*}, and it is
completed by the mass conservation laws \eqref{eq:mass_conserv} and
the entropies
evolution equations \eqref{eq:entrop_cons}.

Actually the SAP ensures conservation principle (as a consequence of
the Noether's theorem, see \cite{BGP23}).
Let $\mathcal E$ be the partial Legendre transform of the Lagrangian $L$ with
respect to the \textit{kinetic} variables $\mathbf{u}$, $\mathrm D_t\alpha$
and $\mathrm D_t a_\mathrm{i}$. It reads
\begin{equation}
  \label{eq:entrop}
 \mathcal E(\rho,K,M,P,\alpha,a_\mathrm{i}, s, s_1,s_2,y)
  = \mathbf{u} K + \mathrm D_t \alpha M + \mathrm D_t a_\mathrm{i} P -L(\mathbf{B})
\end{equation}
or analogously
\begin{equation}
  \label{eq:entrop2}
  \mathcal E(\mathbf{B})
  = L_{kin}(\mathbf{B}) + L_{pot}(\mathbf{B}),
\end{equation}
with notations \eqref{eq:L}.
If the latter formula is more classic, the definition
\eqref{eq:entrop} has the advantage of simplifying the following
computations.

\begin{Prop}[Hyperbolicity]
  \label{prop:noether}
  The energy $\mathcal E$, defined by \eqref{eq:entrop}, satisfies
  the additional scalar conservation equation
  \begin{equation}
    \label{eq:dtE}
    \p_t \mathcal E+ \text{div}((\mathcal E- L^{*,\rho}) \mathbf{u})=0.
  \end{equation}
  If the energy $\mathcal E(\rho,K,M,P,\alpha,a_\mathrm{i}, s, s_1,s_2,y)$ is
  convex, then the system
  \eqref{eq:mass_conserv}-\eqref{eq:entrop_cons}-\eqref{eq:final_set} is
  hyperbolic and it is symmetrizable.
\end{Prop}
\begin{proof}
  Using that $\mathcal E$ is the partial Legendre transform of the Lagrangian
  $L$ with respect to the kinetic variables, it holds (dropping the
  dependency of $\mathcal E$ and $L$ for readability)
  \begin{equation*}
    \begin{aligned}
     \mathrm  D_t \mathcal E &=\mathrm  D_t \left( \sum_{b\in \{\mathbf{u}, \mathrm D_t \alpha, \mathrm D_t a_\mathrm{i}\}}
        b\dfrac{\p L}{\p b} - L\right)\\
      &= \sum_{b\in \{\mathbf{u}, \mathrm D_t \alpha,\mathrm  D_t a_\mathrm{i}\}} \left(\mathrm  D_t b \dfrac{\p
          L}{\p b}+ b \mathrm D_t \left( \dfrac{\p L}{\p b}\right) \right)-
     \mathrm  D_t L.
    \end{aligned}
  \end{equation*}
  Using the notations $K,M$ and $P$, given in \eqref{eq:deltaA1}, and the transport of the specific
  entropies \eqref{eq:entrop_cons} and of the mass fraction \eqref{eq:mass_conserv}, it holds
  \begin{equation*}
    \begin{aligned}
     \mathrm D_t \mathcal E &= \mathbf{u}\mathrm D_t K +\mathrm D_t \alpha \mathrm D_t M + \mathrm D_t a_\mathrm{i} \mathrm D_t P
      - \dfrac{\p L}{\p \rho}\mathrm D_t \rho - \dfrac{\p L}{\p \alpha}\mathrm D_t
      \alpha - \dfrac{\p L}{\p a_\mathrm{i}}\mathrm D_t a_\mathrm{i}.
    \end{aligned}
  \end{equation*}
  Then using the evolution equations \eqref{eq:final_set}, it yields
    \begin{equation*}
    \begin{aligned}
     \mathrm D_t \mathcal E &= \mathbf{u}\cdot \left( -K \text{div}_{\mathbf{x}}(\mathbf{u})+
        \nabla _{\mathbf{x}} L^{*,\rho}\right) + \mathrm D_t \alpha \left( -M
        \text{div}_{\mathbf{x}} (\mathbf{u}) + \dfrac{\p L}{\p \alpha}\right)\\
      &\quad+\mathrm D_t a_\mathrm{i} \left( -P \text{div}_{\mathbf{x}} (\mathbf{u}) +  \dfrac{\p L}{\p
          a_\mathrm{i}}\right)
      -  \dfrac{\p L}{\p \rho}\mathrm D_t \rho -  \dfrac{\p L}{\p \alpha}
     \mathrm D_t\alpha-  \dfrac{\p L}{\p a_\mathrm{i}}\mathrm D_t a_\mathrm{i}\\
      & = -\text{div}_{\mathbf{x}} (\mathbf{u}) \left( K \mathbf{u} + M \mathrm D_t \alpha +
          P \mathrm D_t a_\mathrm{i} - \rho \dfrac{\p L}{\p \rho}\right) +
      \mathbf{u}\cdot \nabla _{\mathbf{x}} \left( \rho \dfrac{\p L}{\p \rho}-L\right).
    \end{aligned}
  \end{equation*}
  Using the definition
  \eqref{eq:entrop} of $\mathcal E$, it gives
  \begin{equation*}
    \begin{aligned}
     \mathrm D_t \mathcal E &= -\text{div}_{\mathbf{x}} (\mathbf{u}) \left(\mathcal E+L-\rho \dfrac{\p L}{\p \rho}\right) +
      \mathbf{u}\cdot \nabla _{\mathbf{x}} L^{*,\rho}\\
      &=  -\text{div}_{\mathbf{x}} (\mathbf{u}) \left( \mathcal E-\nabla_{\mathbf{x}} L^{*,\rho}\right)
        +\mathbf{u}\cdot \nabla_{\mathbf{x}} L^{*,\rho},
    \end{aligned}
  \end{equation*}
  which coincides with \eqref{eq:dtE}. Now if $\mathcal E$ is supposed to be convex
with respect to the variables  $(\rho,K,M,P,\alpha,a_\mathrm{i}, s,
  s_1,s_2,y)$, then it is a Lax entropy of the system which can be
  symmetrized in the sense of Godunov-Mock.
\end{proof}
Note that the sufficient criterion is quite restrictive since the
potential energy $L_{pot}(\mathbf{B})=\rho e(\tilde{\mathbf{B}})$ is not necessarily
strictly convex.

\subsection{Extended final set}
\label{sec:final_set}
We explicit in this paragraph the three equations using the definition
\eqref{eq:L} of the
Lagrangian functional.

\subsubsection{Momentum equation}
\label{sec:momentum-equations}
By the definition \eqref{eq:L*}, the Legendre transform of $L$ with respect to
the density $L^{*,\rho}$ is
\begin{equation}
  \label{eq:L*2}
  L^{*,\rho}(\mathbf{B}) = - \left(\dfrac{m}{2} |\mathrm D_t \alpha|^2
  + \dfrac{\nu}{2} |\mathrm D_t a_\mathrm{i}|^2+ p\right),
\end{equation}
with
\begin{equation}
    \label{eq:pmixt}
  {p = \alpha p_1 + (1-\alpha)p_2
    -  a_\mathrm{i} \gamma_\mathrm{i}},
\end{equation}
is the fluid-interface pressure derived first in
\eqref{eq:mixt_potentials}. Here one uses
$p_1:=p_1\left(\dfrac{\alpha}{y\rho},s_1\right)$ and $p_2:=p_2\left(\dfrac{1-\alpha}{(1-y)\rho},s_2\right)$.
Then the equation on $K=\p L/\p \mathbf{u} =\rho \mathbf{u}$ gives the
momentum equation, namely
\begin{equation*}
  \p_t (\rho \mathbf{u}) + \text{div}_{\mathbf{x}} (\rho \mathbf{u}^\top
  \mathbf{u}) + \nabla _{\mathbf{x}}\left({p} + {\dfrac{m}{2} |D_t\alpha|^2 + \dfrac{\nu}{2} |D_ta_\mathrm{i}|^2}\right) = 0.
\end{equation*}

This equation is similar to the one obtained in
\cite{Drui_these} or \cite{Kokh19}, except that, in this latter
reference, the pressure term
accounts for $\nabla _{\mathbf{x}} \alpha$. When dropping the small scale terms
$\mathrm D_t \alpha$ and $D_t a_\mathrm{i}$, one recovers the momentum flux
derived in \cite{hillairet:hal-03621307} for bubbly flows using an homogenization approach.
The pressure term $p$ comes from the 
potential energy $L_{pot}$ which defines the pressure term in the momentum equation.

\subsubsection{Evolution equations on $\alpha$ and $a_\mathrm{i}$}
\label{sec:alpha-equations}

Since $M=m \mathrm D_t\alpha$ and $P=\nu  \mathrm D_ta_\mathrm{i}$, 
the equations on $M$ and $P$ involve second
order derivatives in time on $\alpha$ and $a_\mathrm{i}$
respectively. Using the definition \eqref{eq:L} of $L$, and relations
\eqref{eq:ei}-\eqref{eq:Ti}, direct computations give
\begin{equation}
  \label{eq:MP}
  \dfrac{\p L}{\p \alpha} = p_1-p_2, \qquad \dfrac{\p L}{\p a_\mathrm{i}} = \gamma_\mathrm{i},
\end{equation}
which lead to 
\begin{equation}
  \label{eq:2nd_order}
  \begin{cases}
    \p_t (\mathrm D_t \alpha) + \text{div}_{\mathbf{x}}  (\mathbf{u}\mathrm D_t
    \alpha) = \dfrac{p_1-p_2}{m},\\
    \p_t (\mathrm D_t a_\mathrm{i}) + \text{div}_{\mathbf{x}}  (\mathbf{u}\mathrm D_t
    a_\mathrm{i}) = \dfrac{\gamma_\mathrm{i}}{\nu}.
  \end{cases}
\end{equation}

Following \cite{Drui_these, Kokh19}, the idea is to
decompose these second order equations into a pair of two first order derivative
in time equations, while introducing
additional unknowns.

For the equation on $M=m \mathrm D_t\alpha$, we fix
\begin{equation}
  \label{eq:w}
  \mathrm D_t\alpha =
  \dfrac{\rho y {w}}{\sqrt{m}}, 
\end{equation}
where $w$ is a new unknown. Then it holds
\begin{equation}
  \label{eq:Dtalpha_w}
  \begin{cases}
    \p_t \alpha+ \mathbf{u}\cdot \nabla_{\mathbf{x}}  \alpha=\dfrac{\rho y {w}}{\sqrt{m}},\\
    \p_t {w} +\mathbf{u}\cdot \nabla_{\mathbf{x}}  w= \dfrac{1}{\sqrt{m}\rho
      y}(p_2-p_1).
  \end{cases}
\end{equation}
Doing so for the equation on   $P=\nu  \mathrm D_ta_\mathrm{i}$, we introduce the
unknown $n$ satisfying
\begin{equation}
  \label{eq:n}
 \mathrm D_ta_\mathrm{i} =\dfrac{\rho y {n}}{\sqrt{\nu}},
\end{equation}
and
it yields
\begin{equation}
  \label{eq:Dtai_n}
  \begin{cases}
    \p_t a_\mathrm{i} + \mathbf{u}\cdot \nabla_{\mathbf{x}}   a_\mathrm{i} = \dfrac{\rho y {n}}{\sqrt{\nu}},\\
    \p_t {n} + \mathbf{u}\cdot \nabla_{\mathbf{x}}  n= \dfrac{\gamma_\mathrm{i}}{\sqrt{\nu}\rho y}.
  \end{cases}
\end{equation}
According to \cite{Kokh19,cordesse20,dibattista:tel}, the equations on the quantities $w$ and $n$, defined in this way, refer
to small scale momentum equations. In that sense the equations on
$\alpha$ and $a_\mathrm{i}$ connect small and large scales.

\subsubsection{Energy equations}
\label{sec:energy-equations}

The transport equations of the specific entropies are not convenient,
especially for numerical computations. We replace them by energy
equations using the Gibbs relations given in Section \ref{sec:therm-modell}.

The total energy equation $\mathcal E=L_{kin}+L_{pot}$
has already been given in \eqref{eq:dtE},
see Proposition \ref{prop:noether}, and its developed form reads
\begin{equation}
  \label{eq:total_nrj}
  \p_t \mathcal E + \text{div}_{\mathbf{x}}  ((\mathcal E+p)\mathbf{u})=0,
\end{equation}
where $p$ refers to the fluid-interface pressure \eqref{eq:pmixt}.

For sake of completness, we provide the phasic (nonconservative)
internal energy equations which read, for $k=1,2$,
\begin{equation}
  \label{eq:dtek}
  \begin{aligned}
    \p_t \left(\alpha_k \rho_k\left( e_k +
        \dfrac{|\mathbf{u}|^2}{2}\right) \right) &+ \text{div}_{\mathbf{x}} \left(
      \left(\alpha_k \rho_k\left( e_k +
          \dfrac{|\mathbf{u}|^2}{2}\right) + \alpha_k p_k\right) \mathbf{u}\right)
    \\
    &=
    \alpha_k p_k \text{div}_{\mathbf{x}} \mathbf{u} -
    y_k\mathbf{u}\cdot\nabla_{\mathbf{x}}  \tilde{p} - p_k  \dfrac{\rho y {w}}{\sqrt{m}},
  \end{aligned}
\end{equation}
where $\tilde{p} = p + \dfrac{m}{2}|\mathrm D_t \alpha|^2 +
\dfrac{\nu}{2}|\mathrm D_t a_\mathrm{i}|^2$.

Now using the transport equations of the specific entropies \eqref{eq:entrop_cons} and the
mass conservation equations \eqref{eq:mass_conserv}, one deduces that the interfacial entropy complies with
\begin{equation}
  \label{eq:dtaisi}
  \p_t (a_\mathrm{i} s_\mathrm{i}) + \text{div}_{\mathbf{x}}  (a_\mathrm{i} s_\mathrm{i} \mathbf{u}) = 0.
\end{equation}
Then combining \eqref{eq:daeI} and \eqref{eq:ei} leads to
the following 
interfacial energy 
evolution equation
\begin{equation}
  \label{eq:dtei}
    \p_t (a_\mathrm{i}e_\mathrm{i}) + \text{div}_{\mathbf{x}}  (a_\mathrm{i} e_\mathrm{i} \mathbf{u})- a_\mathrm{i}\gamma_\mathrm{i} \nabla_{\mathbf{x}} 
  \cdot \mathbf{u} = \gamma_\mathrm{i}\dfrac{\rho y n}{\sqrt{\nu}}.
\end{equation}

\subsubsection{Summary}
\label{sec:summ}

Using the definitions \eqref{eq:w} and \eqref{eq:n} the final set of
equations reads
\begin{equation}
  \label{eq:finalsys}
  \begin{cases}
    \p_t \rho +\text{div}_{\mathbf{x}}  (\rho \mathbf{u}) = 0,\\
    \p_t (\rho y) +\text{div}_{\mathbf{x}}  (\rho y\mathbf{u}) = 0,\\
    \p_t (\rho \mathbf{u}) + \text{div}_{\mathbf{x}}  \left(\rho \mathbf{u}^\top
      \mathbf{u}+\left(p+ \dfrac{m}{2} (\rho y w)^2 +
        \dfrac{\nu}{2} (\rho y n)^2
      \right)\mathbf{Id} \right)= 0,\\
    \p_t \alpha + \mathbf{u} \cdot \nabla_{\mathbf{x}}  \alpha= \dfrac{\rho y
      {w}}{\sqrt{m}},\\
    \p_t a_\mathrm{i} + \mathbf{u} \cdot \nabla_{\mathbf{x}}   a_\mathrm{i} = \dfrac{\rho y
      {n}}{\sqrt{\nu}},\\
    \p_t {w} +\mathbf{u} \cdot\nabla_{\mathbf{x}}  w=
    \dfrac{1}{\sqrt{m}\rho y} (p_1-p_2) ,\\
    \p_t {n} + \mathbf{u} \cdot \nabla_{\mathbf{x}}  n=
    \dfrac{\gamma_\mathrm{i}}{\sqrt{\nu}\rho y},\\
    \mathrm D_t s =  \mathrm D_t s_1 = \mathrm D_t s_2 = 0.
  \end{cases}
\end{equation}

\subsection{Hyperbolicity}
\label{sec:hyperbolicity}
To finish we investigate the eigenstructure of the system \eqref{eq:finalsys},
focusing on its one-dimensional version (with velocity $u$).

For that purpose, let consider the vector $\hat{\mathbf{B}} = (y,
\alpha, a_\mathrm{i}, w, n, s, s_1,s_2) ^\top\in \mathbb R^{8}$ and write the
system \eqref{eq:finalsys} in the following quasilinear form
\begin{equation}
  \label{eq:quasi}
  \p_t
  \begin{pmatrix}
    \rho\\ u\\\hat{\mathbf{B}} 
  \end{pmatrix}
  + \mathbf{C}(\rho, u, \hat{\mathbf{B}} )
    \p_x
  \begin{pmatrix}
    \rho\\ u\\ \hat{\mathbf{B}} 
  \end{pmatrix}
=\mathbf{R},
\end{equation}
where $\mathbf{R}=(0,0,0, \dfrac{\rho y
  {w}}{\sqrt{m}}, \dfrac{\rho y
  {n}}{\sqrt{\nu}},  \dfrac{1}{\sqrt{m}\rho y} (p_1-p_2) ,
\dfrac{\gamma_\mathrm{i}}{\sqrt{\nu}\rho y}, 0,0,0 )^\top\in \mathbb R^{10}$,
and the matrix $\mathbf{C}$ is given by
\begin{equation}
  \label{eq:C}
  \mathbf{C}(\rho, u, \hat{\mathbf{B}} ) =
  \begin{pmatrix}
    u & \rho & \mathbf{0}_{1\times 8}\\
    \dfrac{\hat{p}}{\p \rho} & u &\dfrac{1}{\rho}
    \nabla_{\hat{\mathbf{B}}} \hat p\\
    \mathbf{0}_{8\times 1}& \mathbf{0}_{8\times 1}& u \mathbf{I}_{8 \times 8}
  \end{pmatrix},
\end{equation}
with
\begin{equation*}
  \begin{aligned}
    \hat p(\hat{\mathbf{B}} ) &= 
 \alpha p_1\left( \dfrac{\alpha}{y \rho}, s_1\right) +
    (1-\alpha) p_2\left( \dfrac{1-\alpha}{(1-y) \rho}, s_2\right)\\
    &\qquad+ \dfrac{m}{2} (\rho y w)^2 +
    \dfrac{\nu}{2} (\rho y n)^2.
  \end{aligned}
\end{equation*}
The eigenvalues of $\mathbf{C}$ are 
\begin{equation}
  \label{eq:vp}
  \begin{aligned}
    \lambda_{1,2}= u \pm \rho\sqrt{yc_1^2 + (1-y) c_2^2 +  m (yw)^2+
      \nu (yn)^2},\quad \lambda_{3,\ldots,10} = u,
  \end{aligned}
\end{equation}
where $c_k^2 = \dfrac{\p p_k}{\p \rho_k}(\rho_k,s_k)$ is the speed of
sound of the phase $k=1,2$.
All the eigenvalues are real and the right eigenvectors of
$\mathbf{C}$ constitute a basis of $\mathbb R^{10}$.
This proves again the hyperbolicity of the system.

\medskip

As mention in Remark \ref{rem:dissip}, the SAP depicts reversible
processes only. Hence relaxation has to be set \textit{a posteriori}
according to the second principle.
For instance one may use source terms presented in
\cite{GavrilyukSaurel02} or \cite{dibattista:tel}
to enforce damping due to bubble pulsation. One could also make use of
the paragraph \ref{sec:therm-equil} to design dissipative phase
transition source terms in agreement with the thermodynamical
equilibrium given in Proposition \ref{prop:equi}.

This work has received the financial support from the CNRS grant \textit{D\'efi
  Math\'ema\-tiques France 2030}.
The authors would like to thank S. Kokh and N. Seguin for the fruitful discussions.

\bibliographystyle{plainurl}
\bibliography{BM}

\end{document}